\documentclass[reqno,12pt]{amsart} 
\usepackage{mathrsfs}
\usepackage{amssymb,amsmath,amsthm,amsfonts}
\usepackage[dvipsnames]{xcolor}
\usepackage[
		colorlinks=true,
		linkcolor=blue,
		citecolor=ForestGreen,
		urlcolor=cyan
		]
{hyperref}
\usepackage[inline]{enumitem}  
\usepackage[margin=1in]{geometry}
\let\oldsquare\square
\usepackage{mathabx}
\usepackage{graphicx}
\usepackage{cite}
\usepackage{cases}
\usepackage{cite}

\newcommand{\n}[1]{\left\Vert#1\right\Vert}
\newcommand{\abs}[1]{\left\vert#1\right\vert}
\newcommand{\set}[1]{\left\{#1\right\}}

\newcommand{\R}{\mathbb R}

\newcommand{\C}{\mathbb C}
\newcommand{\Z}{\mathbb Z}
\newcommand{\T}{\mathbb T}
\newcommand{\N}{\mathbb N}

\newcommand{\F}{ {\mathcal F} }

\newcommand{\m}{\mathcal{M}}
\newcommand{\mh}{\widehat{\m}}
\newcommand{\td}{\mathbb{T}^d}

\newcommand{\trs}{\mathbb{T}}
\newcommand{\z}{\Z}

\newcommand{\s}{2\sigma+1}

\def\TagOnRight

\theoremstyle{plain}
\newtheorem{theorem}{Theorem} [section]
\newtheorem{lemma}[theorem]{Lemma}
\newtheorem{corollary}[theorem]{Corollary}

\theoremstyle{remark}
\newtheorem{remark}[theorem]{Remark}

\theoremstyle{definition}
\newtheorem{definition}[theorem]{Definition}

\def\({\left(}
\def\){\right)}
\def\<{\left\langle}
\def\>{\right\rangle}

\numberwithin{equation}{section}

\begin{document}
\title[Norm inflation  for BBM equation]
{Norm inflation for  BBM equation in Fourier amalgam and Wiener amalgam spaces with negative regularity}  

\author[D. Bhimani]{Divyang G. Bhimani}
\address{Department of Mathematics, Indian Institute of Science Education and Research, Dr. Homi Bhabha Road, Pune 411008, India}
\email{divyang.bhimani@iiserpune.ac.in}

\author[S. Haque] {Saikatul Haque}
\address{Harish-Chandra Research Institute, 
India}
\email{saikatulhaque@hri.res.in}

\thanks{}

\subjclass[2010]{35Q55, 35R25 (primary), 42B35 (secondary)}
\keywords{BBM equation; Ill-posedness; Fourier amalgam spaces, Wiener amalgam spaces
  Fourier-Lebesgue spaces; modulation spaces} 
\maketitle
\begin{abstract}
We consider 
Benjamin-Bona-Mahony (BBM)  equation  of the form
$$
u_t+u_x+uu_x-u_{xxt}=0, \quad (x, t)\in \mathcal{M}\times \mathbb R
$$
where  $\mathcal{M}= \mathbb T$ or $\mathbb R.$ We establish 
norm inflation (NI) with infinite loss of regularity at general initial data in  Fourier amalgam and Wiener amalgam spaces with negative regularity.   This  strengthen   several known   NI results  at zero initial data  in $H^s(\mathbb T)$ established by  Bona-Dai (2017) and ill-posedness result established by 
Bona-Tzvetkov (2008) and  Panthee (2011) in $H^s(\mathbb R).$ Our result is sharp with respect to local well-posedness  result of  Banquet-Villamizar-Roa (2021) in  modulation spaces  $M^{2,1}_s(\mathbb R)$ for $s\geq 0$.
\end{abstract}
\section{Introduction}
 We consider 
 Benjamin-Bona-Mahony (BBM)  equation  of the form
\begin{equation}\label{bbm}
\begin{cases}
u_t+u_x+uu_x-u_{xxt}=0\\
u(x,0)=u_0(x)
\end{cases}
\end{equation}
where $u:\mathcal{M}\times \mathbb R\to \mathbb R$ unknown function and $\m= \mathbb T$ or $\mathbb R$.  The BBM  \eqref{bbm} can be written as 
\begin{equation}\label{ebbm}
iu_t= \varphi(D_x)u+ \frac{1}{2} \varphi(D_x)u^2, \quad u(x,0)=u_0(x)
\end{equation}
where $\varphi(\xi)= \frac{\xi}{1+ \xi^2}$, $D_x=\frac{1}{i}\partial_x$ and $\varphi (D_x)$ is the Fourier multiplier operator defined by 
\[  \F[{\varphi (D_x) u}](\xi) = \varphi(\xi) \widehat{u}(\xi).\]

This model BBM \eqref{bbm} is the regularized counterpart of the Korteweg-de Vries (KdV) equation.   This is extensively studied in the literature,  see \cite{alazman2006comparisons,bona1981evaluation,Bona1,pava2011stability,bona1983comparison}.  
BBM equation  \eqref{bbm} is well suited for modeling wave propagation on star graphs.

The aim of this paper is to establish following strong ill-posedness (norm inflation at general initial data with infinite loss of regularity) for \eqref{bbm} in Fourier amalgam $\widehat{w}_s^{p,q}(\m)$ and Wiener amalgam $W^{p, q}_s(\m)$ spaces (to be defined in Section \ref{S2}).
We recall  that 
\begin{equation*} 
   \widehat{w}_s^{p,q}(\m)= \begin{cases}
   \mathcal{F}L_s^q(\m) \  \text{(Fourier-Lebesgue spaces)} \quad \textit{if}  \quad p=q\\
   M^{2,q}_s(\m) \ \text{(modulation spaces)}   \quad \textit{if}  \quad p=2\\
   M^{2,2}_s(\m)=W^{2,2}_s(\m)=H^{s}(\m) \ \text{(Sobolev spaces)} \quad \textit{if} \quad p=q=2\\
   \mathcal{F}L_s^q(\m)=M_s^{p,q}(\m)=W^{p,q}_{s}(\m) \quad \textit{if}    \  \m =\td.
    \end{cases}
    \end{equation*}
These time-frequency  spaces have  proved to be very fruitful in handling various problems in analysis and  made prominent place in  nonlinear dispersive PDEs,  see  e.g.   \cite{wang2007global, wang2011harmonic, bhimani2016functions, B1, B2, B3, ruzhansky2012modulation, kasobenyi, Forlano1}.  
We now state our main theorem.   
\begin{theorem}\label{mt}
Assume that $1\leq p, q\leq \infty, s<0$ and let
\[ X^{p,q}_{s}(\m)=  \begin{cases}\widehat{w}_s^{p,q}(\R) \text{ or }    W_s^{2,q}(\R)\text{ for }\m=\R\\
\F L_s^q(\T)      \text{ for }\m=\T. 
\end{cases} 
  \] 
Then norm inflation with infinite loss of regularity  occurs to \eqref{bbm} everywhere in $X_s^{p,q}(\m)$:
For any $u_0\in X_s^{p,q}(\m)$, $\theta\in\R$ and $\varepsilon>0$ there exist smooth $u_{0,\epsilon} \in X_s^{p,q}(\m)$ and  $T>0$ satisfying  
 \[ \|u_0-u_{0,\epsilon}\|_{X_{s}^{p,q}}< \varepsilon, \quad 0<T< \epsilon\] 
 such that the corresponding smooth solution $u_\epsilon$ to $\eqref{bbm}$ with data $u_{0,\epsilon}$ exists on $[0,T]$ and 
 \[ \|u_\epsilon(T)\|_{X_{\theta}}> \frac{1}{\varepsilon}
 .\] 
In particular,  for any $T>0,$ the solution map $X_s^{p,q}(\m)  \ni u_0\mapsto u \in C([0, T], X_{\theta}^{p,q}(\m))$ for  \eqref{bbm}  is discontinuous everywhere  in $X^{p,q}_{s}(\m)$ for all $\theta\in\mathbb R$.
\end{theorem}

In \cite{Bona1}  Bona and  Tzvetkov proved that \eqref{bbm} is globally well-posed in   $H^s(\mathbb R)$ for $s\geq 0$.  Moreover, they also proved that   \eqref{bbm} is ill-posed for $s<0$ in the sense that the solution map  $u_0\mapsto u(t)$ is not $C^2$ from  $H^s(\mathbb R)$ to  $C([0,T], H^s(\mathbb R))$.  Later in \cite{Panthee1}  Panthee proved that  this is discontinuous at the origin from $H^s(\mathbb R)$ to $\mathcal{D}'(\mathbb R)$.  Recently Bona and Dai in \cite{Dia1} established norm inflation for \eqref{bbm} at zero initial data in $\dot{H^s}(\mathbb T)$ for $s<0.$ We note that Theorem \ref{mt} also hold for the corresponding homogeneous $\dot{X}^{p,q}_s(\m)$ spaces,  see Remark \ref{hs}.
 The particular case of Theorem \ref{mt} strengthen  these results by establishing infinite loss of regularity at every initial data in $H^s(\m)$ for $s<0.$ In \cite[Theorem 1.7]{carlos},  Banquet and Villamizar-Roa  proved that  \eqref{bbm} is locally well-posed in $M_s^{2,1}(\mathbb R)$ for $s\geq 0.$  Thus particular case of Theorem \ref{mt} complement this result by establishing sharp strong ill-posedness in $M^{2,1}_s(\mathbb R)$ for $s<0.$
To the best of authors' knowledge there is no well-posedness result for \eqref{bbm} in Fourier amalgam $\widehat{w}^{p,q} _s \ (p\neq 2)$ (except in $\F L^1(\m)$, see Corollary \ref{uc})  and in $W^{p,q}$ (except in $H^s$) spaces.  

We use Fourier analytic approach to prove Theorem \ref{mt}.  This approach dates back to  Bejenaru and Tao  \cite{bejenaru2006sharp}  to obtain ill-posedness for quadratic NLS and further developed by Iwabuchi in \cite{iwabuchi2015ill}.  Later Kishimoto \cite{kishimoto2019remark} established  norm inflation (NI) for NLS on special domain{\footnote{special domain: $\mathbb R^{d_1}\times \mathbb T^{d_2}, d=d_1+d_2$ and  with nonlinarity:  $\sum_{j=1}^n \nu_ju^{\rho_j}(\bar{u})^{\sigma_j-\rho_j}$ where $\nu_j\in\C$, $\sigma_j \in \mathbb N$, $\rho_j\in\N\cup\{0\}$ with $\sigma_j\geq\max(\rho_j,2)$.} and  Oh  \cite{oh2017remark} established NI at general initial data for cubic NLS.     Recently this approach have been used to obtain  strong ill-posedness for NLW in \cite{Forlano1, bhimani2021norm}.  We refer to \cite[Section 2]{kishimoto2019remark} for detail discussion on this approach.

We now briefly comment and outline the  proof of Theorem \ref{mt}.  We first justify the convergence of series of Picard terms,  the smooth solutions to \eqref{bbm}, in Wiener algebra 
$\mathcal{F}L^1$ (see Corollary \ref{uc}).   This is possible  since linear BBM propagator is unitary on $\mathcal{F}L^1$ and the bilinear operator for the nonlinearity in \eqref{ebbm} is  bounded (see Lemma \ref{est0}).  Then  \eqref{bbm} experience  NI at  general initial data  because (with appropriately chosen initial data close to the given data) one Picard term dominates,  in $X^{p,q}_s-$norm,  rest of the Picard iterate terms in the series   for $s<0$ 
 and also this term  becomes arbitrarily large (see \eqref{rc1}, \eqref{rc2}, \eqref{rc3}).
To this end,    
we  perturb general initial data $u_0$ by $\phi_{0,N}.$  Here,  $\phi_{0, N}$ is defined  on Fourier side by scalar (depends on $N$)
multiplication of characteristic function on  union of two  intervals  obtained by translation of $[-1, 1]$ by $\pm N$ and so the size of support of $\phi_{0, N}$ remain uniform. 
Specifically,  we set 
\[\F{\phi_{0,N}}=R\chi_{I_N}\] where $  I_N=[-N-1,-N+1]\cup[N-1,N+1] $ with $N\gg1,R=R(N)\gg1$ (to be chosen later) and 
\[u_{0,N}=u_0+\phi_{0,N}.\]
Eventually this  $u_{0, N}$  will play a  role of $u_{0,\epsilon}$ in Theorem \ref{mt}. Similar  $\phi_{0,N}$ was used by Bona and Tzvetkov to establish the solution map fails to become $C^2$ in \cite{Bona1} and also by Panthee 
\cite{Panthee1} to conclude  that in fact the solution map is discontinuous.
In \cite{Bona1} the size of the support of $\phi_{0,N}$ in the Fourier side was allowed to vary as $N\to\infty$ with a normalizing constant to make sure $\n{\phi_{0,N}}_{H^s}\sim1$ where as in \cite{Panthee1}, $\F\phi_{0,N}$ is taken as $\chi_{I_N}$ which implies $\n{\phi_{0,N}}_{H^s}\to0$ as $N\to\infty$.
To establish NI with infinite loss of regularity,  we multiply $R=R(N)\gg1$ with the Panthee's choice of $\phi_{0,N}$ to make sure the second Picard iterate $U_2(t)[u_{0,N}]$  has desired property (as mentioned above) and  reduce analysis  on considering single term on $\ell^q-$norm:
$$\|\<n\>^\theta f(n)\|_{\ell_n^q(n=1)}=2^{(s-\theta)/2}\|\<n\>^s f(n)\|_{\ell_n^q(n=1)}\quad  \text{for all} \ \theta \in \mathbb R.$$
as done in NLW case in  \cite{bhimani2021norm}.  We note that finite loss of regularity of NLW was initiated by Lebeau in  \cite{Lebeau05} and infinite loss of regularity for NLS, via geometric optics approach,  by Carles et al.  in \cite{carles2012geometric}.  

 The rest of the paper is organized as follows: In Section \ref{S2}, we recall the definitions of the time-frequency spaces.   In Section \ref{S3}  we establish power series expansion of the solution in $\F L^1$,  by  establishing $\widehat{w}_s^{p,q}$-estimates of the Picard terms for general data. In Section \ref{S4}, we first prove various estimates for of the Picard terms with particular choices of data and this enables us to conclude the proof of Theorem \ref{mt}.

\section{Function Spaces}\label{S2}
 Let $\mathcal{F}$ denotes the Fourier transform and  $\langle \cdot \rangle^s= (1+|\cdot|^2)^{s/2}, s\in \mathbb R.$ Here $\widehat{\m}$ denotes the Pontryagin dual of $\m,$ i.e.,  $\widehat{\m}= \R$ if $\m=\R$ and $\widehat{\m}=\mathbb Z$ if $\m=\trs$. 
 $\mathcal{S}'(\m)$ denotes the space of tempered distributions,  see e.g. \cite[Part II]{RuzhanskyB}  for details. The \textbf{Fourier-Lebesgue spaces $\mathcal{F}L^q_s(\m) \ (1\leq q \leq \infty,  s\in \mathbb R)$} is defined by 
$$\mathcal{F}L^q_s(\m)=\left\{f\in \mathcal{S}'(\m): \mathcal{F}f \  \langle  \cdot  \rangle^s  \in L^q (\widehat{\m})\right \}.$$
In 1980's,  Feichtinger \cite{feichtinger1983modulation} have introduced  the \textbf{modulation spaces  $M^{p,q}_s(\m)$} and \textbf{Wiener amalgam spaces  $W^{p,q}_s(\m)$}  using shrot-time Fourier transform (STFT)\footnote{STFT is also known as windowed Fourier transform and closely related to Fourier-Wigner and  Bargmann transform.  See e.g.  \cite[Lemma 3.1.1]{grochenig2013foundations} and \cite[Proposition 3.4.1]{grochenig2013foundations}. }.  The  STFT  of a  $f\in \mathcal{S}'(\m)$ with
 respect to a window function $0\neq g \in {\mathcal S}(\m)$ is defined by 
\begin{equation*}\label{stft}
V_{g}f(x,y)= \int_{\m} f(t) \overline{T_xg(t)} e^{- 2\pi i y\cdot t}dt,  \  (x, y) \in \m \times \widehat{\m}
\end{equation*}
 whenever the integral exists.  Here, $T_xg(t)=g(tx^{-1})$ is the translation operator on $\m$.  We define modulation $M^{p,q}_s(\m)$ and  and Wiener amalgam spaces $W^{p,q}_s(\m)$, for $1\leq p, q \leq \infty, s\in \mathbb R$ by the norm:  
\[ \|f\|_{M^{p,q}_s}=  \left\| \|V_gf(x,y)\|_{L^p(\m)} \langle y \rangle^s \right\|_{L^q(\widehat{\m})} \quad \text{and} \quad  \|f\|_{W^{p,q}_s(\m)}= \left\| \|V_gf(x,y) \langle y \rangle^s\|_{L^q(\widehat{\m})}  \right\|_{L^p(\m)}. \] 
 The definition of the modulation space  is independent of the choice of 
the particular window function, see  \cite[Proposition 11.3.2(c)]{grochenig2013foundations}.  There is also equivalent characterization of these space via frequency uniform decomposition  (which is quite similar to Besov spaces- where decomposition is dyadic).
To do this,
let   $\rho \in \mathcal{S}(\R)$,  $\rho: \R \to [0,1]$  be  a smooth function satisfying   $\rho(\xi)= 1 \  \text{if} \ \ |\xi|\leq \frac{1}{2} $ and $\rho(\xi)=
0 \  \text{if} \ \ |\xi|\geq  1.$ 
Set $ \rho_n(\xi)= \rho(\xi -n)$
and
$\sigma_{n}(\xi)=
  \frac{\rho_{n}(\xi)}{\sum_{\ell\in\Z^{d}}\rho_{\ell}(\xi)},  n
  \in \Z.$
Then define the frequency-uniform decomposition operators by 
\[\oldsquare_n = \mathcal{F}^{-1} \sigma_n \mathcal{F}. \]
It is known \cite[Proposition 2.1]{wang2007global}, \cite{feichtinger1983modulation} that
\begin{equation*}
\|f\|_{M^{p,q}_s(\m)}\asymp   \left\| \left\lVert
  \oldsquare_nf\right\rVert_{L_x^p(\m)} \langle n \rangle ^{s} \right\|_{\ell^q_n(\Z)}  \quad \text{and} \quad \|f\|_{W^{p,q}_s(\m)}\asymp   \left\| \left\lVert
  \oldsquare_nf \cdot \langle n \rangle ^{s}  \right\|_{\ell^q_n(\Z)}  \right\rVert_{L_x^p(\m)} .
\end{equation*}
Recently in \cite{JH},   Oh and Forlano have introduced \textbf{Fourier amalgam spaces }$\widehat{w}^{p,q}_s(\m) \ (1\leq p, q \leq \infty,  s\in \mathbb R):$
$$\widehat{w}^{p,q}_s(\m)=\left \{ f\in \mathcal{S}'(\m): \|f\|_{\widehat{w}^{p,q}_s}=   \left\| \left\lVert \chi_{n+Q_1}(\xi)\F f(\xi)\right\rVert_{L_\xi^p(\mh)} \langle n \rangle^s \right\|_{\ell^q_n(\z)}< \infty \right\},
$$
where $Q_1=(-\frac{1}{2},\frac{1}{2}].$
The  \textbf{homogeneous spaces $\dot{X}^{p,q}_s(\m)$} corresponding to above spaces can be defined by replacing Japanese bracket $\langle \cdot \rangle^{s}$ by $|\cdot|^s$ in their definitions.

\section{Local well-posedness in  Wiener algebra $\mathcal{F}L^1$}\label{S3}
The integral version of \eqref{ebbm} is given by
\begin{equation}\label{ibbm}
u(t)=U(t)u_0-\frac{i}{2}\int_0^tU(t-\tau)\varphi(D_x)u^2(\tau)d\tau
\end{equation}
where $\mathcal{F}U(t) \varphi(D_x) u(\xi)= e^{it \varphi (\xi)} \varphi (\xi) \mathcal{F}u (\xi)$ and   $U(t)u_0(x)= \mathcal{F}^{-1} (e^{it\varphi(\xi)}  \mathcal{F}u_0(\xi))(x)$ is the unique solution to the  linear problem
\[
iu_t= \varphi(D_x)u,  \ \ u(x,0)=u_0(x); \quad   (x, t)\in \m \times \mathbb R.
\]
Let us define the operator $\mathcal{N}$ given by
\[
\mathcal{N}(u,v)(t)=\int_0^tU(t-\tau)\varphi(D_x)(uv)(\tau)d\tau.
\]
\begin{definition}[Picard iteration]\label{imd} 
For $u_0\in L^2(\R^d),$  define
$U_1[u_0](t)= S(t)u_0$ and for $k\geq2$
\[U_k[u_0](t)=-\frac{i}{2}\sum_{{k_1, k_2 \geq 1}\atop{k_1+k_2 =k}}\mathcal{N}\left(  S_{k_1}[u_0],S_{k_2}[u_0]\right)(t). \]
\end{definition}
\begin{lemma}\label{est0}
Let $1\leq p, q \leq \infty,  s, t \in \mathbb R.$ Then  we have
\begin{enumerate}
\item $\n{U(t)u_0}_{\widehat{w}^{p,q}_s}=\n{u_0}_{\widehat{w}^{p,q}_s}$
\item $\n{\mathcal{N}(u,v)(t)}_{\widehat{w}^{p,q}_s}\leq\int_0^t\n{u(\tau)}_{\F L^1}\n{v(\tau)}_{\widehat{w}^{p,q}_s}d\tau\leq t\n{u}_{L\infty((0,t),\F L^1)}\n{v}_{L^\infty((0,t),\widehat{w}^{p,q}_s)}.$
\end{enumerate}
\end{lemma}
\begin{proof}
Note that
\begin{align*}
\n{U(t)u_0}_{\widehat{w}^{p,q}_s}=\left\| \left\lVert \chi_{n+Q_1}(\xi) e^{it\varphi(\xi)}\F u_0 (\xi)\right\rVert_{L_\xi^p(\mh)} (1+|n|^2)^{s/2} \right\|_{\ell^q_n(\z)}=\n{u_0}_{\widehat{w}^{p,q}_s}.
\end{align*}
Using the fact $\abs{\varphi}\leq1$   we have
\begin{align*}
\n{\mathcal{N}(u,v)(t)}_{\widehat{w}^{p,q}_s}
&=\left\| \left\lVert \chi_{n+Q_1}(\xi)\int_0^te^{i(t-\tau)\varphi(\xi)}\varphi(\xi)(\F u\ast\F v)(\xi)(\tau)d\tau\right\rVert_{L_\xi^p(\mh)} \<n\>^s \right\|_{\ell^q_n(\z)}\\
&\leq\left\|\int_0^t \left\lVert \chi_{n+Q_1}(\xi)e^{i(t-\tau)\varphi(\xi)}\varphi(\xi)(\F u\ast\F v)(\xi)(\tau)\right\rVert_{L_\xi^p(\mh)}d\tau \<n\>^s \right\|_{\ell^q_n(\z)}\\
&\leq\left\|\int_0^t \left\lVert \chi_{n+Q_1}(\xi)(\F u\ast\F v)(\xi)(\tau)\right\rVert_{L_\xi^p(\mh)}d\tau \<n\>^s \right\|_{\ell^q_n(\z)}\\
&\leq\int_0^t\left\| \n{\F u(\xi)(\tau)}_{L_\xi^1(\mh)}\left\lVert \chi_{n+Q_1}(\xi)\F v(\xi)(\tau)\right\rVert_{L_\xi^p(\mh)} \<n\>^s \right\|_{\ell^q_n(\z)}d\tau\\
&= \int_0^t\n{u(\tau)}_{\F L^1}\n{v(\tau)}_{\widehat{w}^{p,q}_s}d\tau.
\end{align*}
%
\end{proof}

\begin{lemma}[See \cite{kishimoto2019remark}] \label{DS}
  \label{it0}  Let $\{b_k\}_{k=1}^{\infty}$ be a sequence of nonnegative  real numbers such that 
\[ b_k \leq C  \sum_{{k_1, k_2 \geq
      1}\atop{k_1+ k_2 =k}}  b_{k_1} b_{k_2} \quad \forall\ k \geq 2.\]
Then  we have 
$b_k \leq b_1 C_0^{k-1}$, for all $k \geq 1$, where $C_0=
  \frac{2\pi^2}{3} C  b_1$.  
\end{lemma}

%

\begin{lemma}\label{est}
 There exists $c>0$ such that for all $t>0$ and $k\geq2$, we have 
\begin{align*}
\n{U_k[u_0](t)}_{\widehat{w}_s^{p,q}}\leq (ct)^{{k-1}}\n{u_0}_{\F L^1}^{k-1}\n{u_0}_{\widehat{w}_s^{p,q}}.
\end{align*}
\end{lemma}
\begin{proof}[{\bf Proof}] 
Let $\{b_k\}$ be a  sequence of nonnegative  real numbers such that 
\[ b_1=1\quad \text{and} \quad  b_k = \frac{1}{k-1}  \sum_{{k_1,k_2 \geq
      1}\atop{k_1+k_2 =k}}  b_{k_1}b_{k_2}
  \quad \forall\ k \geq 2.\]
  By Lemma \ref{it0}, we have $b_k\leq c_0^{k-1}$ for some $c_0>0$. In view of this, it is enough to prove the following claim: 
\begin{align*}
\n{U_k[u_0](t)}_{\widehat{w}_s^{p,q}}\leq b_kt^{{k-1}}\n{u_0}_{\F L^1}^{k-1}\n{u_0}_{\widehat{w}_s^{p,q}}. 
\end{align*}
By Definition \ref{imd},  Lemma \ref{est0} and using the fact $\frac{\abs{\xi}}{1+\xi^2}\leq1$, we have 
\begin{align}\label{est01}
\n{ U_k[u_0](t)}_{\widehat{w}_s^{p,q} }
\leq  \sum_{{k_1, k_2 \geq 1}\atop{k_1+k_2 =k}}\int_0^t\n{ U_{k_1}[u_0](\tau)}_{\F L^1}\n{ U_{k_2}[u_0](\tau)}_{\widehat{w}^{p,q}_s} d\tau
\end{align}
Thus, we have 
 \begin{align*}
 \n{ U_2[u_0](t)}_{\widehat{w}_s^{p,q} }\leq t\n{ U[u_0]}_{L^\infty((0,t),\F L^1)}\n{ U[u_0]}_{L^\infty((0,t),\widehat{w}^{p,q}_s)}=t\n{ u_0}_{\F L^1}\n{ u_0}_{\widehat{w}^{p,q}_s}
\end{align*}
 Hence,  the claim is true for $k=2$ as $b_2=1$. 
Assume that the result is true up to the label $(k-1)$. Then from \eqref{est01}, we obtain
\begin{align*}
\n{ U_k[u_0](t)}_{\widehat{w}_s^{p,q}}
&\leq \sum_{{k_1, k_2 \geq 1}\atop{k_1+k_2 =k}}b_{k_1}b_{k_2}\int_0^t \tau^{{k_1-1}}\n{u_0}_{\F L^1}^{k_1}\times\tau^{{k_2-1}}\n{u_0}_{\F L^1}^{k_2-1}\n{u_0}_{\widehat{w}_s^{p,q}}d\tau\\
&=b_kt^{{k-1}}\n{u_0}_{\F L^1}^{k-1}\n{u_0}_{\widehat{w}_s^{p,q}}.
\end{align*}
Thus the claim is true at the level $k$. This completes the proof.
\end{proof}

\begin{corollary}\label{uc}
If $0<T\ll  M^{-1},$ then  for any  $u_0 \in \F L^1$ with  $\|u_0\|_{\F L^1} \leq M$, there exists a unique solution $u\in C([0, T], \F L^1(\m))$ to the integral equation \eqref{ibbm} associated with \eqref{ebbm}, given by 
 \begin{eqnarray}\label{pse}
 u= \sum_{k=1}^{\infty} U_k[u_0]
 \end{eqnarray}
 which converges absolutely in $C([0, T], \F L^1(\m)).$
\end{corollary}
\begin{proof}
Define 
\[\Psi(u)(t)=U(t)u_0-\frac{i}{2}\mathcal{N}(u,u)(t).\]
By Lemma \ref{est0},  we have 
\begin{align*}
\n{\Psi(u)}_{C([0,T],\F L^1)}&\leq \n{u_0}_{\F L^1}+T\n{u}_{C([0,T],\F L^1)}^2,\\
\n{\Psi(u)-\Psi(v)}_{C([0,T],\F L^1)}&\lesssim T\max\(\n{u}_{C([0,T],\F L^1)},\n{v}_{C([0,T],\F L^1)}\)\n{u-v}_{C([0,T],\F L^1)}.
\end{align*}
 Then considering the ball
\[B_{2M}^T=\set{\phi\in C([0,T],\F L^1):\n{\phi}_{C([0,T],\F L^1)}\leq 2M}\]
 with 
 $TM\ll1$ we find a fixed point of $\Psi$ in $B_{2M}^T$ and hence a solution to \eqref{ibbm}. This completes the proof of the first part of the lemma.

For the second part we note that in view of Lemma \ref{est}, the series \eqref{pse} converges absolutely if $0<T\ll M^{-1}$. Then for $\epsilon>0$, there exists $j_1$ such that for all $j\geq j_1$
one has 
\begin{equation}\label{s1}
\n{u-u_j}_{C([0,T],\F L^1)}<\epsilon
\end{equation}
where \[u=\sum_{k=1}^\infty U_k[u_0],\quad\text{and}\quad u_j=\sum_{k=1}^j U_k[u_0].\]
Note that $u,u_j\in B_{2M}^T$ for all $j$ as $0<T\ll M^{-1}$. Using the continuity of $\Psi$ on $B_{2M}^T$ we find $j_2$ such that for all $j\geq j_2$ 
\begin{equation}\label{s2}
\n{\Psi(u)-\Psi(u_j)}_{C([0,T],\F L^1)}<\epsilon.
\end{equation}
Note that 
\begin{align*}
u_j-\Psi(u_j)&=\sum_{k=1}^j U_k[{u}_0]-U(t){u}_0+\frac{i}{2}\mathcal{N}(u_j,u_j)\\
&=\sum_{k=2}^j U_k[{u}_0]+\frac{i}{2}\sum_{1\leq k_1,k_2\leq j}\mathcal{N}(U_{k_1}[{u}_0],U_{k_2}[{u}_0])\\
&=\frac{i}{2}\sum_{k=j+1}^{2j} \sum_{{1\leq k_1,k_2\leq j}\atop{k_1+k_2=k}}\mathcal{N}(U_{k_1}[{u}_0],U_{k_2}[{u}_0])=-\sum_{k=j+1}^{2j} U_{k,j}[{u}_0]
\end{align*}
where we set 
\[U_{k,j}[{u}_0]=-\frac{i}{2}\sum_{{1\leq k_1,k_2\leq j}\atop{k_1+k_2=k}}\mathcal{N}(U_{k_1}[{u}_0],U_{k_2}[{u}_0]).\]
Note that $U_{k,j}$ has a less number of term in the sum above compared to that of $U_k$. Therefore proceeding as in the proof of Lemma \ref{est}, one achieves same estimates for $U_{k,j}$.
Therefore using $0<T\ll M^{-1}$
\begin{align*}
\n{u_j-\Psi(u_j)}_{C([0,T],\F L^1)}&\leq\sum_{k=j+1}^{2j}\n{U_{k,j}[{u}_0]}_{C([0,T],\F L^1)}\\
&\leq\sum_{k=j+1}^{2j}c^{k-1}T^{k-1}\n{{u}_0}_{\mathcal{F}L^1}^k\\
&\leq M\sum_{k=j+1}^\infty(cTM)^{k-1}\leq2 M(cMT)^j.
\end{align*}
Then there exists $j_3$ such that for  $j\geq j_3$ one has
\begin{equation}\label{s3}
\n{u_j-\Psi(u_j)}_{C([0,T],\F L^1)}<\epsilon.
\end{equation}
Therefore from \eqref{s1}, \eqref{s2} and \eqref{s3} one has
\[\n{u-\Psi(u)}_{C([0,T],\F L^1)}<3\epsilon.\]
Thus $u$ is the required fixed point for $\Psi$. 
\end{proof}

\section{Proof of Theorem \ref{mt}}\label{S4}
We first prove NI with infinite loss of regularity at general data in $\F L^1(\m)\cap{X}_s^{p,q}(\m)$. Subsequently, for general data in ${X}_s^{p,q}(\m)$ we use the density of $\F L^1(\m)\cap{X}_s^{p,q}(\m)$ in $\mathcal{X}_s^{p,q}(\m)$ ($s<0$). So let us begin with $u_0\in\F L^1(\m)\cap{X}_s^{p,q}(\m)$.  Now define $\phi_{0,N}$ on $\m$ via   the following relation
 \begin{equation}\label{kid}
\F{\phi_{0,N}}(\xi)=R\chi_{I_N}(\xi) \quad (\xi \in \widehat{\m})
\end{equation}
where  $ I_N=[-N-1,-N+1]\cup[N-1,N+1]$ and  $ N\gg1,R\gg1$ to be chosen later.  Note  that 
\begin{align}\label{1}
\n{\phi_{0,N}}_{\widehat{w}_s^{p,q}}\sim RN^{s}.
\end{align}
Let us set \begin{equation}\label{data}
u_{0,N}=u_0+\phi_{0,N}
\end{equation}

\begin{lemma}[See Lemma 3.6. in \cite{kishimoto2019remark}]\label{kms0}
There exists $C>0$ such that for  $u_0$ satisfying \eqref{kid} and $k\geq 1$, we have 
\[\left|  \operatorname{supp}\F{ U_k[\phi_{0,N}]} (t)\right| \leq C^k
  ,\quad \forall t\geq 0.\]
\end{lemma}
\subsection{Estimates in $\widehat{w}_s^{p,q}(\m)$}
\begin{lemma}\label{d0}Let $u_0$  be given by \eqref{kid}, $s\leq0$ and $1\leq p, q \leq  \infty.$ Then there exists $C$ such that
  \begin{enumerate}
      \item $\n{{u}_{0,N}-{u}_0}_{\widehat{w}_s^{p,q}}\lesssim RN^s$\label{1d1}
      \item $\|U_1[{u}_{0,N}](t)\|_{\widehat{w}_s^{p,q}} \lesssim  1+RN^s$\label{2d1}
      \item $\|U_2[{u}_{0,N}](t)-U_2[{\phi}_{0,N}](t)\|_{\widehat{w}_s^{p,q}}\lesssim tR$\label{3d1}
      \item $\|U_k[\vec{u}_{0,N}](t)\|_{\widehat{w}_s^{p,q}} \lesssim C^kR^{k}t^{k-1}.$\label{4d1}
  \end{enumerate} 
\end{lemma}
\begin{proof}[{\bf Proof}]
\eqref{1d1} follows from \eqref{1}.
By Lemma \ref{est0} and \eqref{1}, we have  $\|U_1[\phi_{0,N}](t)\|_{\widehat{w}_s^{p,q}}=\n{\phi_{0,N}}_{\widehat{w}_s^{p,q}} \sim RN^s$. Then \eqref{2d1} follows by using triangle inequality. 
By Lemma \ref{est} and \eqref{1}, we obtain
\begin{align*}
\|{U_k[{\phi}_{0,N}](t)}\|_{\widehat{w}_s^{p,q}}
&\leq \sup_{\xi \in \mh} |  \F{U_k[{\phi}_{0,N}]} (t, \xi)| \mu_{\mh}(\operatorname{supp}\F{ U_k[{\phi}_{0,N}]} (t))^{1/p}\n{\<n\>^s}_{\ell^q(n\in{\rm supp }\ \F U_k[\phi_{0,N}](t))}\\
& \lesssim (ct)^{k-1}    R^{k} \n{\<n\>^s}_{\ell^q(n\in{\rm supp }\ \F U_k[\phi_{0,N}](t))}.
\end{align*}
where ${\mu_{\mh}}(A)$ denotes the $\mh$-measure of the set $A$.  Since $s\leq0$, for any bounded set $D\subset \R$, we have 
\[\| \langle n\rangle^s \|_{\ell^q (  n \in D )}\leq \| \langle n\rangle^s \|_{\ell^q ( n \in B_D )}\]
where $B_D\subset \R$ is the Interval centred at the origin with $|D|=|B_D|.$  
In view of this and Lemma \ref{kms0}, 
we obtain 
\begin{equation*}
\| \langle n\rangle^s \|_{\ell^q( \operatorname{supp}
  \widehat{U_k}[\phi_{0,N}](t))}\leq
\| \langle n\rangle^s \|_{\ell^q ( \{ |n| \leq C^{k/d} \})}
\lesssim   C^{k/q}. 
\end{equation*}
Therefore 
\begin{align}\label{5d1}
\|{U_k[{\phi}_{0,N}](t)}\|_{\widehat{w}_s^{p,q}}\leq C^kt^{k-1}R^{k}.
\end{align}
Now observe that 
\begin{align*}
I_{k}(t):=&U_k[{u}_{0,N}](t)-U_k[{\phi}_{0,N}](t)\\
&=\sum_{{k_1, k_2 \geq 1}\atop{k_1+k_2=k}}\mathcal{N}(U_{k_1}[{u}_0+{\phi}_{0,N}],U_{k_2}[{u}_0+{\phi}_{0,N}])-\mathcal{N}(U_{k_1}[\phi_{0,N}],U_{k_2}[\phi_{0,N}])\\
&=\sum_{{k_1,k_2 \geq 1}\atop{k_1+k_2 =k}}\sum_{({\psi_1},{\psi}_{2})\in\mathcal{C}}\mathcal{N}(U_{k_1}[{\psi}_1],U_{k_2}[{\psi}_{\s}])
\end{align*}
where $\mathcal{C}=\{{u}_0,{\phi}_{0,N}\}^{2}\setminus\{({\phi}_{0,N},{\phi}_{0,N})\}$. Observe that $\mathcal{C}$ has atleast one coordinate as $\vec{u}_0$.
Using Lemma \ref{est0} and the proof of Lemma \ref{est}
 it follows that
\begin{align*}
\|I_k(t)\|_{\widehat{w}_s^{p,q}}
& \lesssim \sum_{{k_1, k_2\geq 1}\atop{k_1+k_2=k}}\sum_{(v_1,v_2)\in \mathcal{C}}\int_0^t\|U_{k_1}[{v}_1](\tau)\|_{\widehat{w}_s^{p,q}}\|U_{k_2}[v_2](\tau)\|_{{\F L^1}}d\tau\\
&\leq (2^2-1)2\|{u}_0\|_{\widehat{w}_s^{p,q}} (\|{u}_0\|_{{\F L^1}}^{k-1}+\|{\phi}_{0,N}\|_{{\F L^1}}^{k-1})\int_0^t\tau^{k-2}d\tau\sum_{{k_1,k_2 \geq 1}\atop{k_1+k_2 =k}}b_{k_1} b_{k_2}\\
&\leq 12b_kt^{k-1}R^{k-1}\|{u}_0\|_{\widehat{w}_s^{p,q}}\leq C^kt^{k-1}R^{k-1}\|{u}_0\|_{\widehat{w}_s^{p,q}}.
\end{align*}
as $R\gg1$. Note that \eqref{3d1} is the particular case $k=2$ and \eqref{4d1} follows using the above and \eqref{5d1}.
\end{proof}

\begin{lemma}\label{d2}  Let $u_0$  be given by \eqref{kid}, $1\leq p \leq  \infty$, $s\in\R$ and   $0<T\ll 1,$ then we   have\[ \|U_2 [{\phi}_{0,N}] (T)\|_{\widehat{w}_s^{p,q}}\geq \left\| \left\lVert \chi_{n+Q_1}(\xi)\F U_2 [{\phi}_{0,N}] (T)(\xi)\right\rVert_{L_\xi^p} \<n\>^{s} \right\|_{\ell^q(n=1)} 
\gtrsim  R^2T. \]
\end{lemma}
\begin{proof}[{\bf Proof}]
For notational convenience we put  $$\Gamma_\xi=\{(\xi_1,\xi_2):\xi_1+\xi_2=\xi\}\quad \text{and}\quad\Phi=c(-\varphi(\xi)+\varphi(\xi_1)+\varphi(\xi_2)).$$
 Using the symmetry of set $\Sigma$, we have
\begin{align}\label{dspl}
\F U_2[u_0] (T)(\xi)
&=\int_0^Te^{ic(T-t)\varphi(\xi)}\varphi(\xi)\(\F U_1(t)u_0\ast\F{U_1(t)u_0}\) dt \nonumber \\
&=\int_0^Te^{ic(T-t)\varphi(\xi)}\varphi(\xi)\left[e^{i ct\varphi}\F u_0\ast e^{i ct\varphi}\F u_0\right] (\xi) dt \nonumber\\
&=e^{icT\varphi(\xi)}\varphi(\xi)R^2\int_0^T\int_{\Gamma_\xi}{e^{it\Phi}}\chi_{I_N}(\xi_1)\chi_{I_N}(\xi_2)d\Gamma_{\xi} dt.
\end{align}
Note that with $\xi_1+\xi_2=\xi$ one has $$\Phi(\xi,\xi_1,\xi_2)=c\frac{\xi \xi_1 \xi_2 (\xi^2-\xi_1\xi_2+3)}{(1+\xi_1^2)(1+ \xi_2^2) (1+\xi^2)}$$
and hence for $\xi\in [\frac{1}{2},1]$ and $\xi_1, \xi_2\in I_N,$  we have  $|\Phi|\sim 1.$  Hence, $|t\Phi|\ll 1$  for $0<t\ll 1$ and so 
\[ \text{Re} \int_0^T e^{it\Phi} dt \geq \frac{T}{2}. \] 
Also note that $\abs{\varphi(\xi)}\gtrsim 1$ for $\xi\in [\frac{1}{2},1]$. Thus, we have for $\xi\in [\frac{1}{2},1]\subset I_N+I_N$
\begin{align}\label{M2}
\abs{\F U_{2} [u_0] (T)(\xi)}&\gtrsim R^2T\int_{\Gamma_\xi}\chi_{I_N}(\xi_1)\chi_{I_N}(\xi_2)d\Gamma_{\xi} 
= R^2T\chi_{I_N}\ast\chi_{I_N}(\xi)\geq R^2T\chi_{[-1,1]}
\end{align}
as $\chi_{a+[-1,1]}\ast\chi_{b+[-1,1]}\geq\chi_{a+b+[-1,1]}$. The above pointwise estimate immediately  gives the desired estimate:
\[\|U_2[{\phi}_{0,N}] (T)\|_{\widehat{w}_s^{p,q}}\geq\left\| \left\lVert \chi_{n+Q_1}(\xi)\F U_2 [\vec{\phi}_{0,N}] (T)(\xi)\right\rVert_{L_\xi^p([\frac{1}{2},1]\cap\mh)} \<n\>^{s} \right\|_{\ell^q(n=1)} 
\gtrsim TR^2
\]provided $0<T\ll1$.
\end{proof}

\subsection{Estimates in $W_s^{2,q}(\R)$}
\begin{lemma}[inclusion]\label{INC} Let $p,q,q_1,q_2\in[1,\infty]$ and $s\in\R$. Then
\begin{enumerate}
\item $\|f\|_{W_s^{2,q}}\leq \|f\|_{\widehat{w}_s^{2,q}}$ if $q\leq 2$\label{Inclsn1}
\item $\|f\|_{W_s^{p,q_1}}\lesssim\|f\|_{W_s^{p,q_2}}$ if $q_1\geq q_2$\label{Inclsn2}
\end{enumerate}
\end{lemma}
\begin{proof}
\eqref{Inclsn1} is a consequence of Minkowski inequality and Plancherel theorem  whereas \eqref{Inclsn2} follows from the fact that $\ell^{q_2}\hookrightarrow\ell^{q_1}$ if $q_1\geq q_2$. 
\end{proof}

\begin{lemma}\label{d0'}Let $u_0$  be given by \eqref{kid}, $s\leq0$ and $1\leq p \leq  \infty.$ Then there exists $C$ such that
  \begin{enumerate}
      \item $\n{{u}_{0,N}-{u}_0}_{W_s^{2,q}}\lesssim RN^s$
      \item $\|U_1[{u}_{0,N}](t)\|_{W_s^{2,q}} \lesssim  1+RN^s$
      \item $\|U_2[{u}_{0,N}](t)-U_2[{\phi}_{0,N}](t)\|_{W_s^{2,q}}\lesssim tR$
      \item $\|U_k[\vec{u}_{0,N}](t)\|_{W_s^{2,q}} \lesssim C^kR^{k}t^{k-1}$
  \end{enumerate} 
\end{lemma}
\begin{proof}
By Lemma \ref{INC}, we have 
\begin{align*}
\|\vec{u}_{0,N}-\vec{u}_0\|_{W^{2,q}_s} \lesssim\begin{cases}
\|\vec{u}_{0,N}-\vec{u}_0\|_{\widehat{w}^{2,q}_s}\lesssim RN^s&\text{ for }q\in[1,2]\\
 \|\vec{u}_{0,N}-\vec{u}_0\|_{W^{2,2}_s} \lesssim RN^s&\text{ for }q\in(2,\infty]
\end{cases}
\end{align*}  
using Lemma \ref{d0} \eqref{1d1}. Similarly, the other estimates also follow from Lemmata \ref{d0}.
\end{proof}

\begin{lemma}\label{d2'}  Let $u_0$  be given by \eqref{kid}, $1\leq p \leq  \infty$, $s\in\R$ and   $0<T\ll 1,$ then we   have\[ \|U_2 [{\phi}_{0,N}] (T)\|_{W_s^{2,q}}\geq \left\| \left\lVert\F^{-1} \sigma_n\F U_2 [{\phi}_{0,N}] (T)(\xi)\<n\>^{s} \right\rVert_{\ell^q(n=1)} \right\|_{L_\xi^2} 
\gtrsim  R^2T. \]
\end{lemma}
\begin{proof}
Note that using Plancherel theorem and \eqref{M2} we have
\begin{align*}
\|U_2 [{\phi}_{0,N}] (T)\|_{W^{2,q}_s}&\geq\left\| \lVert\F^{-1} \sigma_n\F U_2[{\phi}_{0,N}] (T)(x)\<n\>^{s}\rVert_{\ell^q(n=e_1)}  \right\|_{L_x^2} \\
&=2^{s/2}\left\|  \sigma_{e_1}\F U_2[\vec{\phi}_{0,N}] (T)(\xi)  \right\|_{L_\xi^2}\gtrsim R^2T.
\end{align*}
This completes the proof.
\end{proof}

\begin{proof}[{\bf Proof of Theorem \ref{mt}}]
We first consider the case $\mathcal{X}_s^{p,q}=\widehat{w}_s^{p,q}$.
If the initial data $u_{0,N}$ satisfies \eqref{data}, Corollary \ref{uc} guarantees existence of the solution to  \eqref{ibbm} and
the power series expansion in $\F L^1$ upto time  $
TR\ll1
$ (as $R\gg1$).
By Lemma \ref{d0}, we obtain
\begin{equation}\label{series}
\sum_{k=3}^\infty\n{U_k[u_{0,N}](T)}_{\widehat{w}_s^{p,q}}\lesssim T^2R^3
\end{equation}
provided $
TR\ll1.
$
Note that 
\begin{align*}
\n{u_N(T)}_{\widehat{w}_\theta^{p,q}}\geq\|{\n{\chi_{n+Q_1}\F u_N(T)}_{L^p}\<n\>^\theta}\|_{\ell^q(n=1)}\sim_{\theta,s}\n{\n{\chi_{n+Q_1}\F u_N(T)}_{L^p}\<n\>^s}_{\ell^q(n=1)}
\end{align*}
Using Corollary \ref{uc} and triangle inequality, we have 
\begin{align*}
&\n{u_N(T)}_{\widehat{w}_\theta^{p,q}}\\
&\gtrsim \n{\n{\chi_{n+Q_1}\F U_2[u_{0,N}](T)}_{L^p}\<n\>^s}_{\ell^q(n=1)}-c\bigg(\n{\n{\chi_{n+Q_1}\F U_1[u_{0,N}](T)}_{L^p}\<n\>^s}_{\ell^q(n=1)}\\&\ \ \ +\sum_{k=3}^\infty\n{\n{\chi_{n+Q_1}\F U_k[u_{0,N}](T)}_{L^p}\<n\>^s}_{\ell^q(n=1)}\bigg)\\
&\gtrsim \n{\n{\chi_{n+Q_1}\F U_2[u_{0,N}](T)}_{L^p}\<n\>^s}_{\ell^q(n=1)}-c\n{U_1 [\vec{u}_{0,N}] (T)}_{\widehat{w}_s^{p,q}}-c\sum_{k=3}^\infty\n{U_k [\vec{u}_{0,N}] (T)}_{\widehat{w}_s^{p,q}}.
\end{align*}
Let $m\in\mathbb{N}$. In order to ensure 
$ \|u_N(T)\|_{\widehat{w}_\theta^{p,q}} \gtrsim  \| U_2[\vec{u}_{0,N}](T)\|_{\widehat{w}_s^{p,q}}  \gg m$
we rely on the conditions 
\begin{numcases}{\n{\n{\chi_{n+Q_1}\F U_2[u_{0,N}](T)}_{L^2}\<n\>^s}_{\ell^q(n=1)}\gg}
\|U_1[\vec{u}_{0,N}](T)\|_{\widehat{w}_s^{p,q}},\label{rc1}\\
 \sum_{k=2}^{\infty}\left\| U_k\vec{u}_{0,N}](T) \right\|_{\widehat{w}_s^{p,q}},\label{rc2}\\
m.\label{rc3}
\end{numcases}
Thus to establish NI with infinite loss of regularity at $u_0$ in $\widehat{w}_s^{p,q}$ we claim that it is 
enough to have the followings
\begin{enumerate}
\item $CRN^s<1/m$\label{i}
\item \label{ii} $TR\ll1$
\item $TR^2\gg m$\label{iii}
\item $TR^2\gg T^2R^3$\label{iv} $\Leftrightarrow$ \eqref{ii} 
\item $0<T\ll 1$\label{v}
\end{enumerate} as $N\to\infty$.
Note that \eqref{i} ensures $\|{u_0-u_{0,N}}\|_{\widehat{w}_s^{p,q}}<1/m$ whereas \eqref{ii} ensures the convergence of the infinite series in view of Lemma \ref{d0}.
In order to use Lemma \ref{d2} we need \eqref{v}. In order to prove \eqref{rc2}, in view of Lemma \ref{d2}  and \eqref{series}, we need  \eqref{iv}. Condition \eqref{iii} implies \eqref{rc3} using Lemma \ref{d0} \eqref{3d1} and Lemma \ref{d2}. In order to prove \eqref{rc1}, we need \eqref{i} and \eqref{iii} by using Lemma \ref{d0} \eqref{2d1} and Lemma \ref{d2}. Thus,   it follows that 
\[\|{u_0-u_{0,N}}\|_{\widehat{w}_s^{p,q}}<1/m\quad\text{and}\quad\|u_N(T)\|_{\widehat{w}_\theta^{p,q}}> m.
\]
Hence,  the result is established. We shall now choose $R$ and $T$ as follows: 
\[
 R=N^r\quad  \text{and} \quad T=N^{-\epsilon}.
\]
where $r,\epsilon$ to be chosen below. 
Therefore, it is enough to check
\begin{align*}
CRN^s=CN^{r+s}<1/m,\quad
TR=N^{-\epsilon+r}\ll1,\quad
TR^2=N^{-\epsilon+2r}\gg m,\quad
T&=N^{-\epsilon}\ll1.
\end{align*}
Thus we only need to achieve:
\begin{itemize}
\item $r+s<0$
\item $-\epsilon+r<0$
\item $-\epsilon+2r>0$
\item $\epsilon>0$
\end{itemize} and take $N$ large enough.
Let us concentrate on the choice of $\epsilon>0$ first. Note that the second and third conditions in the above are equivalent with
\begin{align*}
r<\epsilon<2r.
\end{align*}To make room for $\epsilon$ we must have $
r>0.
$ 
On the other hand, since we want $\epsilon>0$ we also require\[
r>0.
\]
 Thus $r$ must satisfy
\[0<r<-s\]where the later condition comes from first condition. Thus it is enough to choose
$$r=-\frac{s}{3},\quad\epsilon=-\frac{s}{2}$$ which will satisfy all the above four conditions. Hence the result follows.

For the case $\mathcal{X}_s^{p,q}=W_s^{2,q}$ we use same argument as above: 
 Note that using Lemmata \ref{d0'}, \ref{d2'} we have
\begin{align*}
&\n{u_N(T)}_{W_{\theta}^{2,q}}\\
&\geq\n{\n{ \oldsquare_n u_N(T)\<n\>^\sigma}_{\ell^q(n=1)}}_{L^2}\sim_{ \theta,s}\n{\n{ \oldsquare_n u_N(T)\<n\>^s}_{\ell^q(n=1)}}_{L^2}\\
&\gtrsim\n{\n{\oldsquare_n U_2[u_{0,N}](T)\<n\>^s}_{\ell^q(n=1)}}_{L^2}- c\|U_1[u_{0,N}](T)\|_{W_s^{2,q}}- c \sum_{k=3}^{\infty}\left\| U_k[\vec{u}_{0,N}](T) \right\|_{W_s^{2,q}}\\
&\gtrsim\n{\n{\oldsquare_n U_2[u_{0,N}](T)\<n\>^s}_{\ell^q(n=1)}}_{L^2}\gg m.
\end{align*}
and $\n{\vec{u}_{0,N}-\vec{u}_0}_{W_s^{2,q}}<1/m$ provided we choose $R,N,T$ as in the case of $\widehat{w}_s^{p,q}$
\end{proof}

\begin{remark}\label{hs} It is easy to check that our proof  of main results will work even if   we replace the weight $\<\cdot\>^s$ by $|\cdot|^s$ in the function spaces involved.  Since the  analysis will be similar we omit the details.   We  just note that 
 as $\langle n \rangle^s \asymp |n|^s$ for large $n$,  we  have $\|\phi_{0, N}\|_{ \dot{ \widehat{w}}^{p,q}_s} \asymp RN^s$ where $\phi_{0,N}$ is as in \eqref{kid}.
Moreover,  it should work with any weight $n\mapsto (\omega(n))^s (s<0)$ which is decreasing in $|n|$ and behaves like $|n|^s$ as $n\to\infty$.
\end{remark}

\noindent
{\textbf{Acknowledgement}:} D.  G.  B is thankful to DST-INSPIRE (DST/INSPIRE/04/2016/001507) for the research grant.   S. H.  acknowledges Dept of Atomic Energy, Govt of India, for the financial support and Harish-Chandra Research Institute for the research facilities provided. 
\bibliographystyle{siam}
\bibliography{bbm}
\end{document}